\documentclass[12pt,twoside, draft]{article}
\usepackage{ifthen}
\usepackage{amsmath}
\usepackage{amsfonts}
\usepackage{latexsym}
\usepackage{amsthm}
\usepackage{amssymb}

\date{}

\begin{document}

%\centerline{\bf Journal's Title, Vol. x, 200x, no. xx, xxx - xxx}

\centerline{}

\centerline{}

\centerline {\Large{\bf Zeros distribution of gaussian entire functions}}

\centerline{}

\centerline{\bf {A. O. Kuryliak}}

\centerline{}

\centerline{Department of Mechanics and Mathematics,}

\centerline{Ivan Franko National University of L'viv, Ukraine}

\centerline{kurylyak88@gmail.com}

\centerline{}

\centerline{\bf {O. B. Skaskiv}}

\centerline{}

\centerline{Department of Mechanics and Mathematics,}

\centerline{Ivan Franko National University of L'viv, Ukraine}

\centerline{matstud@franko.lviv.ua}

\centerline{}

\newtheorem{Theorem}{\quad Theorem}[section]

\newtheorem{Definition}[Theorem]{\quad Definition}

\newtheorem{Corollary}[Theorem]{\quad Corollary}

\newtheorem{Lemma}[Theorem]{\quad Lemma}

\newtheorem{Example}[Theorem]{\quad Example}

\begin{abstract}
In this paper we consider a random entire function of the form
$f(z,\omega )=\sum\nolimits_{n=0}^{+\infty}\xi_n(\omega )a_nz^n,$
where $\xi_n(\omega )$ are independent standard\break complex gaussian
random variables and $a_n\in\mathbb{C}$ satis\-fy the relations\break
$\varlimsup\limits_{n\to+\infty}\sqrt[n]{|a_n|}=0$ and $ \#\{n\colon a_n\neq0\}=+\infty.$
We investigate asymptotic properties of the
probability $P_0(r)=P\{\omega\colon  f(z,\omega )$ has
no zeros inside $r\mathbb{D}\}.$ Denote
$
p_0(r)=\ln^-P_0(r),\ N(r)=\#\{n\colon \ln (|a_n|r^n)>0\},$ $  s(r)=\sum_{n=0}^{+\infty}\ln^+(|a_n|r^{n}).
$
Assuming that $a_0\neq0$ we prove that
\begin{gather*}
0\leq\varliminf_{\begin{substack} {r\to+\infty \\ r\notin E}\end{substack}}\frac{\ln(p_0(r)- s(r))}{\ln s(r)},\ \varlimsup_{\begin{substack} {r\to+\infty \\ r\notin E}\end{substack}}\frac{\ln(p_0(r)- s(r))}{\ln s(r)}\leq\frac12,\\
 \lim_{\begin{substack} {r\to+\infty \\ r\notin E}\end{substack}}\frac{\ln(p_0(r)- s(r))}{\ln N(r)}=1.
\end{gather*}
where $E$ is a set of finite logarithmic measure. Remark that the previous inequalities are sharp.
Also we give an answer to open question from \cite[p. 119]{nishry 5}.
\end{abstract}



\section{Introduction}
One of the problems on random functions is
investi\-gation of value distribution of there functions and also
asymptotic  properties of the probability of absence of zeros
in a disc (``hole probability''). These problems were considered in papers of
J.~E.~Littlewood and A.~C.~Offord (\cite{littlewood offord}--\cite{offord1993}),
M.~Sodin and B. Tsirelson (\cite{sodin tsirelson1}--\cite{sodin tsirelson3}),
Yu. Peres and B.~Virag (\cite{peres virag}, \cite{peres virag arxiv}),
 P.~V.~Filevych and M.~P.~Mahola (\cite{fil mag 2010}--\cite{fil mag sub}),
 M.~Sodin (\cite{sodin math res 2000}--\cite{sodin arxiv}),
  F.~Nazarov,  M.~Sodin and  A.~Volberg (\cite{nazarov sodin volberg transportation}, \cite{nazarov sodin volberg}),
M.~Krishnapur (\cite{krishnapur}),
A.~Nishry (\cite{nishry 1}--\cite{nishry 5}) and many others.

So, in \cite{sodin tsirelson3} it was considered a random entire function of the form
\begin{equation}\label{1}
\psi(z,\omega)=\sum_{k=0}^{+\infty}\xi_k(\omega)
\frac{z^k}{\sqrt{k!}},
  \end{equation}
where $\{\xi_k(\omega)\}$ are independent complex valued random variables
with the density function
$$
p_{\xi_k}(z)=\frac{1}{\pi}e^{-|z|^2},\ z\in\mathbb{C}, \ k\in\mathbb{Z}_+.
$$
We denote such a distribution by $\mathcal{N}_{\mathbb{C}}(0,1).$

Let us denote by $n_{\psi}(r,\omega)$ the counting function of zeros of the function $\psi(z,\omega)$ in $r\mathbb{D}=\{z\colon |z|<r\}.$
Then (\cite{sodin tsirelson3}) for any $\delta\in(0,1/4]$ and all $r\geq1$
we have
$$
P\Bigl\{\omega\colon\Bigl|\frac{n(r,\omega)}{r^2}-1\Bigl|\geq\delta\Bigl\}
\leq\exp(-c(\delta)r^4),
$$
where the constant $c(\delta)$ depends only on $\delta.$
Also in \cite{sodin tsirelson3} it was investigated the probability of absence of zeros
of the function $\psi(z,\omega),$
$$
P_0(r)=P\{\omega\colon\psi(z,\omega)\neq0\mbox{ inside }r\mathbb{D}\}.
$$
In particular, it was proved in \cite{sodin tsirelson3} that there exist constants $c_1, c_2>0$ such that
$$
\exp(-c_1r^4)\leq P_0(r)\leq\exp(-c_2r^4)\quad (r\geq1).
$$
Also in \cite{sodin tsirelson3} the authors put the following question:\
{\it Does there exist the limit} $$\lim\limits_{r\to+\infty}\frac{\ln^{-}P_0(r)}{r^4}\ ?$$

We find the answer to this question in \cite{nishry 1}. For the function $\psi(z,\omega)$ it was
proved that
$$
\lim_{r\to+\infty}\frac{\ln^{-}P_0(r)}{r^4}=\frac{3e^2}{4}.
$$

In \cite{nishry 1} it was proved that if all of $\xi_n(\omega)\colon \xi_n(\Omega)\subset K,$ where $K\subset\mathbb{C}$
and $0\not\in K$
then there exists $r_0(K)<+\infty$ such that $\psi(z,\omega)$ must vanish somewhere in the disc $r_0\mathbb{D}.$

For the function of the form (\ref{1}) one can fix the disc of radius $r$
and ask for the asymptotic behaviour of $P\{\omega\colon  n_{\psi}(r,\omega )\geq m\}$ as $m\to+\infty$.
So in  \cite{krishnapur} it was proved, that for any $r>0$ we get
$$
\ln P\{\omega\colon  n_{\psi}(r,\omega )\geq m\}=-\frac12m^2\ln m(1+o(1))\quad  (m\to+\infty).
$$

Very large deviations of zeros of function (\ref{1}) were also considered in \cite{nazarov sodin volberg}.
There we find such a relation
$$
\lim_{r\to+\infty}\frac{\ln\ln\Bigl(\frac1{P\{\omega\colon |n_{\psi}(r,\omega )-r^2|>r^{\alpha}\}}\Bigl)}{\ln r}=
\begin{cases}
  2\alpha-1, & \frac12\leq\alpha\leq1;\\
  3\alpha-2, & 1\leq\alpha\leq2;\\
  2\alpha, & \alpha\geq 2.
\end{cases}
$$

More generally, in \cite{nishry 4, nishry 5} it was considered gaussian entire functions
of the form
$$
f(z,\omega )=\sum_{n=0}^{+\infty}\xi_k(\omega )a_nz^n,
$$
where $a_0\neq0,\ n\in\mathbb{Z}_+,\ \varlimsup\limits_{n\to+\infty}\sqrt[n]{|a_n|}=0.$
If $\varepsilon>0$, then there exists (\cite{nishry 4, nishry 5}) a~set of finite logarithmic measure $E\subset(1,+\infty)$
($\int_E\frac{dr}{r}<+\infty$) such that for all $r\in(1,+\infty)\setminus E$ we obtain
\begin{equation}\label{aim}
 s(r)-s^{1/2+\varepsilon}(r)\leq p_0(r)\leq s(r)+s^{1/2+\varepsilon}(r), \  s(r)=\sum_{n=0}^{+\infty}\ln^+(a_nr^n).
\end{equation}

One can find similar results for gaussian analytic functions in the unit disc  in \cite{krishnapur},
 \cite{peres virag}, \cite{peres virag arxiv}, \cite{sodin arxiv}.

Also in \cite[p. 119]{nishry 5}, it was formulated the following question:\
{\it Is the error term in inequa\-lity~(\ref{aim}) optimal for a regular
sequence of the coefficients $\{a_n\}$?}
The {\bf aim} of this paper is to answer this question.

\section{Notation}
In this section we consider the functions of the form
\begin{equation}\label{6}
f(z,\omega )=\sum_{n=0}^{+\infty}\xi_n(\omega) a_nz^n,
\end{equation}
where $\xi_n(\omega)\in\mathcal{N}_{\mathbb{C}} (0,1)$ and $a_n\in\mathbb{C},\ n\in\mathbb{Z}_+$ such that
$\#\{n\colon a_n\neq0\}=+\infty,\ \varlimsup\limits_{n\to+\infty}\sqrt[n]{|a_n|}=0.$

In this paper we study asymptotic behaviour of
$$
p_0(r)=\ln^-P\{\omega\colon  n_f(r,\omega )=0\}
$$
as $r\to+\infty$
for random entire functions of the form (\ref{6}).

For $r>0, \ \delta>0$ denote
\begin{gather*}
  \mathcal{N}'=\{n\colon a_n=0\},\
  \mathcal{N}(r)=\{n\colon \ln (|a_n|r^n)>0\},\\
  \mathcal{N}_{\delta}(r)=\{n\colon \ln (|a_n|r^n)>-\delta n\},\
      N(r)=\#\mathcal{N}(r),\
  N_{\delta}(r)=\#\mathcal{N}_{\delta}(r),
  \\
   m(r)=\sum_{n\in\mathcal{N}(r)}n,\
  m_{\delta}(r)=\sum_{n\in\mathcal{N}_{\delta}(r)}n,\
  s(r)=2\sum_{n\in\mathcal{N}(r)}\ln(|a_n|r^n),
  \\
  \mu_f(r)=\max\{|a_n|r^n\colon n\in\mathbb{Z}_+\},\
  \nu_f(r)=\max\{n\colon \mu_f(r)=|a_n|r^n\},\\
  M_f(r)=\max\{|f(z)|\colon |z|\leq r\},\
  S_f^2(r)=\sum_{n=0}^{+\infty}|a_n|^2r^{2n}.
\end{gather*}

\section{Auxiliary lemmas}
\begin{Lemma}[Borel-Nevanlinna, \cite{goldberg_ostrovky}]\sl\label{l1}
Let $u(r)$ be a nondecreasing continuous function on $[r_0; +\infty)$
and $\lim_{r\to+\infty}u(r)=+\infty,$ and
 $\varphi(u)$ be a continuous nonincreasing
positive function defined on $[u_0;+\infty)$ and
\begin{itemize}
\item[1)] $u_0=u(r_0);$
\item[2)] $\lim\limits_{u\to+\infty}\varphi(u)=0;$
\item[3)] $\int_{u_0}^{+\infty}\varphi(u)du<+\infty.$
\end{itemize}
Then for all $r\geq r_0$ outside a set $E$ of finite measure
we have
$$
u\{r-\varphi(u(r))\}>u(r)-1.
$$
\end{Lemma}

We need the following elementary corollary of this lemma.
\begin{Lemma}\sl\label{l2}
There exists a set $E\subset(1;+\infty)$ of finite logarithmic measure
 such that for all
$r\in(1;+\infty)\setminus E$ we obtain
$$
m(re^{-\delta})>em(r)\exp\{-2\sqrt{\ln m(r)}\},\
m(re^{\delta})<em(r)\exp\{2\sqrt{\ln m(r)}\},
$$
where $\delta=\frac1{2\ln m(r)}.$
\end{Lemma}

\begin{Lemma}\sl\label{l3}
Let $\varepsilon>0.$  There exists a set $E\subset(1;+\infty)$ of finite logarithmic measure
 such that for all
$r\in(1;+\infty)\setminus E$ we have
\begin{equation}\label{7}
N(r)<s^{1/2}(r)\exp\{(1+\varepsilon)\sqrt{\ln s(r)}\}.
\end{equation}
\end{Lemma}

\begin{proof}
Remark that (see also \cite{nishry 4})
$$
N_{-\delta}(r)=\#\{n\colon |a_n|r^n\geq e^{\delta n}\}=
\#\{n\colon |a_n|(re^{-\delta})^n\geq1\}=N(re^{-\delta})$$ and
$$
m(r)=\sum_{n\in\mathcal{N}(r)}n\geq\sum_{n=0}^{ N(r)-1}n=\frac{( N(r)-1) N(r)}{2}>\frac{N^2(r)}{e}
$$
for $r>r_0,$ where $N(r_0)>4.$ So we obtain for $r\in(1;+\infty)\setminus E$
\begin{gather*}
\frac{ s(r)}{2}=\sum_{n\in\mathcal{N}(r)}\ln(|a_n|r^n)\geq
\sum_{n\in\mathcal{N}_{-\delta}(r)}\ln(|a_n|r^n)
\geq\sum_{n\in\mathcal{N}_{-\delta}(r)}n\delta=\\
=\delta m(re^{-\delta})>\frac{e}{2\ln m(r)}m(r)\exp\{-2\sqrt{\ln m(r)}\}.
\end{gather*}
Then
$$
\ln s(r)>1+\ln m(r)-2\sqrt{\ln m(r)}-\ln\ln m(r)
$$
and for $r>r_2$ we get $\ln m(r)<2\ln s(r).$ So for any $\varepsilon>0$
\begin{gather*}
   s(r)>em(r)\exp\{-2\sqrt{\ln m(r)}-\ln\ln m(r)\}>\\
  >e\frac{N^2(r)}{e}\exp\{-2\sqrt{(1+\varepsilon)\ln s(r)-\ln((1+\varepsilon)\ln s(r))}\}>\\
  >N^2(r)\exp\{-(2+2\varepsilon)\sqrt{\ln s(r)}\},
\end{gather*}
as $r\to+\infty$ outside some set of finite logarithmic measure.
\end{proof}

Let us note that the exponent $1/2$ in the inequality (\ref{7})
can not be replaced by a smaller number.

\begin{Lemma}\sl
There exist a random entire function of the form (\ref{6})
 and a set $E\subset(1;+\infty)$ of finite logarithmic measure
 such that for all
$r\in(1;+\infty)\setminus E$ we have
$$
 N(r)>\frac{\sqrt{ s(r)}}{\ln^3 s(r)}.
$$
\end{Lemma}
\begin{proof}
Choose
$$
f(z)=1+\sum_{n=1}^{+\infty}\frac{z^n}{(\frac{n}{2})^{\frac{n}{2}}}.
$$
 Then the function $y(n)=\ln a_n=-\frac{n}{2}\ln(\frac{n}{2})$ is concave function
 and the sequence $\{a_n\}$ is log-concave(\cite{nishry 2}, \cite{skaskiv kuryliak zeros}).
Also by Stirling's formula we obtain
\begin{gather*}
M_f(r)=1+\sum_{n=1}^{+\infty}\frac{r^n}{(\frac{n}{2})^{\frac{n}{2}}}>
1+\sum_{n/2=1}^{+\infty}\frac{r^n}{(\frac{n}{2})^{\frac{n}{2}}}=
1+\sum_{m=1}^{+\infty}\frac{r^{2m}}{m^m}>
1+\sum_{m=1}^{+\infty}\frac{r^{2m}}{m!e^m}=\\
=\exp\Bigl\{\frac{r^2}{e}\Bigl\},\ \ln M_f(r)>\frac{r^2}{e}.
\end{gather*}
By Wiman-Valiron's theorem there exists a
set $E$ of finite logarithmic measure such that
for all $r\in(1;+\infty)\setminus E$ we get
$
\ln\mu_f(r)+\ln\ln\mu_f(r)>\ln M_f(r)>{r^2}/{e},\ \ln\mu_f(r)>{r^2}/{2e}
$
 and finally
 $$
 \frac{r^2}{2e}<\ln\mu_f(r)=\ln a_{\nu}+\nu_f(r)\ln r,\
 \nu_f(r)>\frac{1}{\ln r}\Bigl(\frac{r^2}{2e}-\ln a_{\nu}\Bigl)>r,\ r\to+\infty.
 $$

Therefore, outside a
set $E$ of finite logarithmic measure we get (\cite{nishry 2})
\begin{gather*}
 s(r)<2( N(r)+1)\ln\mu_f(r)<\ln^2\mu_f(r)(\ln\ln\mu_f(r))^2=\\
=\ln^3r\frac{\ln^2\mu_f(r)}{\ln^2r}\frac{(\ln\ln\mu_f(r))^2}{\ln r}<
\ln^3r\nu_f^2(r)\ln^2\nu_f(r)<\\
<\nu_f^2(r)\ln^5\nu_f(r)<N^{2}(r)\ln^5 N(r)<N^{2}(r)\ln^5 s(r),\\
 N(r)>\sqrt{\frac{ s(r)}{\ln^5 s(r)}}>\frac{\sqrt{ s(r)}}{\ln^3 s(r)}.
\end{gather*}

\end{proof}

Also we will use the following lemma.
\begin{Lemma}\sl\label{l4}
Let $\{\zeta_n(\omega )\}$ be a sequence of independent identically distributed
random variables,
such that $M|\zeta_n|<+\infty$ and $ M(\frac1{|\zeta_n|})<+\infty,\ n\in\mathbb{Z}_+.$
Then
$$
P\Bigl\{\omega\colon (\exists N^*(\omega ))(\forall n>N^*(\omega ))\ \Bigl[\frac{1}{n}\leq|\zeta_n(\omega )|\leq n\Bigl]\Bigl\}=1.
$$
\end{Lemma}

\begin{proof}[Proof]
  Let $F_{|\zeta|}(t)=F_{|\zeta_n|}(t)$ be the distribution function of the random variable $|\zeta_n|,\ n\in\mathbb{Z}_+.$

  Denote $B_m=\{\omega\colon  |\zeta_m(w)|\geq m\},\ m\in \mathbb{Z}_+.$ Then
  \begin{gather*}
  \sum_{m=1}^{+\infty}P\{\omega\colon  |\zeta_m(w)|\geq m\}=\sum_{m=1}^{+\infty}
  \int\limits_{|t|\geq m}dF_{|\zeta|}(t)=
  \sum_{m=1}^{+\infty}\sum_{s=m}^{+\infty}\int\limits_{|t|\in[s,s+1)}dF_{|\zeta|}(t)=
  \\
  =\sum_{s=1}^{+\infty}\sum_{m=1}^{s}\int\limits_{|t|\in[s,s+1)}dF_{|\zeta|}(t)=
  \sum_{s=1}^{+\infty}{s}\int\limits_{|t|\in[s,s+1)}dF_{|\zeta|}(t)\leq\\
  \leq\sum_{s=1}^{+\infty}\int\limits_{|t|\in[s,s+1)}|t|dF_{|\zeta|}(t)=M|\zeta|<+\infty.
  \end{gather*}

  Therefore we obtain $\sum_{m=1}^{+\infty}P(B_m)<+\infty.$ So, by the Borel-Cantelli lemma
  only finite quantity of the events $B_n$ may occur. Then
  $$
  P(A_1)=P\Bigl\{\omega\colon (\exists N_1^*(\omega ))(\forall n>N_1^*(\omega ))\ \Bigl[|\zeta_n(\omega )|\leq n\Bigl]\Bigl\}=1.
  $$

  Since $M(\frac1{|\zeta|})<+\infty,$ we get similarly for random variable $\frac1{|\zeta(\omega )|}$
   \begin{gather*}
   P(A_2)=P\Bigl\{\omega\colon (\exists N^*_2(\omega ))(\forall n>N^*_2(\omega ))\ \Bigl[ \frac1{|\zeta_n(\omega )|}\leq n\Bigl]\Bigl\}=\\
   =P\Bigl\{\omega\colon (\exists N_2^*(\omega ))(\forall n>N^*_2(\omega ))\ \Bigl[|\zeta_n(\omega )|\geq \frac1{n}\Bigl]\Bigl\}=1.
  \end{gather*}
Finally,
$$
P(A_1\cap A_2)=P\Bigl\{\omega\colon (\exists N^*(\omega ))(\forall n>N^*(\omega ))\ \Bigl[\frac{1}{n}\leq|\zeta_n(\omega )|\leq n\Bigl]\Bigl\}=1,
$$
where $N^*(\omega )=\max\{N^*_1(\omega ), N^*_2(\omega )\}.$
\end{proof}

The random variables $\xi_n\in\mathcal{N}_{\mathbb{C}}(0,1),\ n\in\mathbb{Z}_+$ satisfy conditions of\break Lemma~\ref{l4}.
So, we get the following result.
\begin{Lemma}\sl\label{l44}
Let $ \xi_n\in\mathcal{N}_{\mathbb{C}}(0,1),\ n\in\mathbb{Z}_+.$
Then
$$
P\Bigl\{\omega\colon (\exists N^*(\omega ))(\forall n>N^*(\omega ))\ \Bigl[\frac{1}{n}\leq|\xi_n(\omega )|\leq n\Bigl]\Bigl\}=1.
$$
\end{Lemma}

\section{Upper and lower bounds for $p_0(r)$}

\begin{Theorem}\sl\label{t1}
Let $\varepsilon>0$ and $f(z,\omega)$ be random
 entire function of the form~(\ref{6}) with $a_0\neq0.$
 There exists a set $E\subset(1;+\infty)$ of finite logarithmic measure
 such that for all
$r\in(1;+\infty)\setminus E$ we have
\begin{equation}\label{8}
p_0(r)\leq s(r)+ N(r)\exp\{(2+\varepsilon)\sqrt{\ln N(r)}\}.
\end{equation}
\end{Theorem}
\begin{proof}[Proof.]
Similarly as in \cite{nishry 4}, we consider the event
$
\Omega_1=\cap_{i=1}^4A_i
$, where
\begin{gather*}
A_1=\Bigl\{\omega\colon |\xi_0(\omega )|\geq\frac{2eN^{1/3}(r)\exp\{2\sqrt{\ln N(r)}\}}{|a_0|}\Bigl\},\\
A_2=\Bigl\{\omega\colon (\forall n\in\mathcal{N}(r)\setminus\{0\}) \Bigl[|\xi_n(\omega )|\leq\frac{1}{|a_n|r^{n}N^{2/3}(r)}\Bigl]\Bigl\},\\
A_3=\Bigl\{\omega\colon (\forall n\in\mathcal{N}_{\delta}(r)\setminus(\mathcal{N}(r)\cup\{0\})) \Bigl[|\xi_n(\omega )|\leq\frac{1}{N^{2/3}(r)}\Bigl]\Bigl\},\\
A_4=\Bigl\{\omega\colon (\forall n\not\in\mathcal{N}_{\delta}(r)\cup\mathcal{N}'\cup\{0\}) \Bigl[|\xi_n(\omega )|\leq n\Bigl] \Bigl\},\ \delta=\frac{1}{2\ln N(r)}.
\end{gather*}

 If $\Omega_1$ occurs, then for $r\not\in E$ we obtain
 \begin{gather*}
|\xi_0(\omega )a_0|-\Biggl|\sum_{n=1}^{+\infty}\xi_n(\omega )a_nr^n\Biggl|
\geq 2eN^{1/3}(r)\exp\{2\sqrt{\ln N(r)}\}-\\
-\sum_{n\in\mathcal{N}(r)}\frac{|a_n|r^n}{|a_n|r^{n}N^{2/3}(r)}
-\sum_{n\in\mathcal{N}_{\delta}(r)\setminus\mathcal{N}(r)}\frac{|a_n|r^n}{N^{2/3}(r)}-
\sum_{n\not\in\mathcal{N}_{\delta}(r)\cup\mathcal{N}'}ne^{-n\delta}>\\
>2eN^{1/3}(r)\exp\{2\sqrt{\ln N(r)}\}-
\sum_{n\in\mathcal{N}_{\delta}(r)}\frac{1}{N^{2/3}(r)}-\\
-\int_{1}^{+\infty}xe^{-\delta x}dx
>2eN^{1/3}(r)\exp\{2\sqrt{\ln N(r)}\}-\\
-N^{1/3}(r)
-eN^{1/3}(r)\exp\{2\sqrt{\ln N(r)}\}-8\ln^2 N(r)>0
 \end{gather*}
 as $r\to+\infty,$ because
 \begin{gather*}
 \int_{1}^{+\infty}xe^{-\delta x}dx=\frac{e^{-\delta}}{\delta^2}(\delta+1)<\frac{2}{\delta^2}=8\ln^2 N(r).
 \end{gather*}

So, we proved that first term dominants the sum of all the other terms
inside $r\mathbb{D},$ i.e.
\begin{equation}\label{8b}
|\xi_0(\omega ) a_0|>\Biggl|\sum_{n=1}^{+\infty}\xi_n(\omega) a_nz^n\Biggl|.
\end{equation}
If $\Omega_1$ occurs then the function $f(z,\omega )$ has no zeros inside $r\mathbb{D}$. Now we find a lower bound for the
probability  of the event $\Omega_1$.
 \begin{gather*}
   P(A_1)=\exp\Bigl\{-\frac{4e^2N^{2/3}(r)\exp\{4\sqrt{\ln N(r)}\}}{|a_0|^2}\Bigl\},
 \\
   P(A_2)\geq\prod_{n\in\mathcal{N}(r)}\frac{1}{2|a_n|^2r^{2n}N^{4/3}(r)}
   =\prod_{n\in\mathcal{N}(r)}\frac{1}{2|a_n|^2r^{2n}}\times\\
   \times\exp\{-N(r)\ln(N^{4/3}(r))\}
   =\exp\Bigl\{-s(r)-\frac43 N(r)\ln N(r)- N(r)\ln2\Bigl\},
   \\
   P(A_3)\geq\prod_{n\in\mathcal{N}(re^{\delta})}\frac{1}{2N^{4/3}(r)}\geq
   \exp\Bigl\{-N(re^{\delta})\ln(2N^{4/3}(r))\Bigl\}\geq\\
   \geq\exp\Bigl\{-e{N}(r)\exp\{2\sqrt{{N}(r)}\}\ln(2N^{4/3}(r))\Bigl\},\\
   P(A_4)=P\{\omega\colon  (\forall n\not\in\mathcal{N}_{\delta}(r)\cup\mathcal{N}'\cup\{0\})|\xi_n(\omega)|<n\}\geq\\
  \geq 1-\sum_{n\not\in\mathcal{N}_{\delta}(r)\cup\mathcal{N}'\cup\{0\}}e^{-n^2}>\frac12, \ r\to+\infty.
 \end{gather*}
 It follows from $\Omega_1\subset\{\omega\colon  n(r,\omega )=0\}$ that for any $\varepsilon>0$ and for every $r\in[r_0,+\infty)\setminus E$
 \begin{gather*}
   p_0(r)=\ln^-P\{\omega\colon  n(r,\omega )=0\}\leq\ln^-P(\Omega_1)=\sum_{n=1}^4\ln^-P(A_n)
   \leq\\
   \leq \ln2+\frac{4e^2N^{2/3}(r)\exp\{4\sqrt{\ln N(r)}\}}{|a_0|^2}
   +s(r)+2N(r)\ln N(r)+N(r)\ln2+\\
   +e{N}(r)\exp\{2\sqrt{{N}(r)}\}\ln(2N^{4/3}(r))
   \leq s(r)+{N}(r)\exp\{(2+\varepsilon)\sqrt{{N}(r)}\}
 \end{gather*}
\end{proof}

A random entire function of the form
\begin{equation}\label{9}
g(z,\omega )=\sum_{n=0}^{+\infty}e^{2\pi i\theta_n(\omega )}a_nz^n
\end{equation}
where independent random variables $\theta_n(\omega )$ are uniformly distributed on $(0,1),$
was considered in \cite{fil mag sub}.
For such functions there were proved the following statements.

\begin{Theorem}[\cite{fil mag 2010}]\sl\label{tb}
Let $g(z,\omega )$ be a random entire function of the form (\ref{9}). Then for
$r>r_0$ we obtain
$$
N_{g}(r,\omega )\leq\frac{1}{2e}+\ln S_{g}(r),
$$
where
$$
N_{g}(r,\omega )=\frac{1}{2\pi}\int_0^{2\pi}\ln|g(re^{i\alpha},\omega )|d\alpha-\ln |a_k|,
$$
and $k=\min\{n\in\mathbb{Z}_+\colon a_n\neq0\}.$
\end{Theorem}

\begin{Theorem}[\cite{fil mag sub}]\sl\label{ta}
There exists an absolute constant $C>0$ such that for a function $g(z,\omega )$
of the form (\ref{9}) almost surely we have
\begin{equation}\label{star}
\ln S_{g}(r)\leq N_{g}(r,\omega )+C\ln N_{g}(r,\omega ),\ r_0(\omega )\leq r<+\infty.
\end{equation}
\end{Theorem}

\begin{Corollary}\sl
Let $(\zeta_n(\omega ))$ be a sequence of independent identically distributed random variables
such that for any $n\in\mathbb{N}$ random variable $\mathop{\rm arg}\zeta_n(\omega )$
is uniformly distributed on $[-\pi,\pi)$ and $M|\xi_n|<+\infty$, $ M(\frac1{|\xi_n|})<+\infty,\ n\in\mathbb{Z}_+.$
Then there exists an absolute constant $C>0$
such that for every random function of the form
$
f(z,\omega )=\sum_{n=0}^{+\infty}\zeta_n(\omega )a_nz^n
$
we get almost surely
$$
\frac{1}{2\pi}\int_0^{2\pi}\ln|f(re^{i\alpha},\omega )|d\alpha-\ln |a_k\zeta_k(\omega )|\geq
\ln S_f(r,\omega)-(C+1)\ln\ln S_f(r,\omega)
$$
for $r_0(\omega )\leq r<+\infty$ and $k=\min\{n\in\mathbb{Z}_+\colon a_n\neq0\}.$
 \end{Corollary}

Since random variables ${\mathop{\rm arg}}\ \xi_n(\omega) ($here $\xi_n(\omega)\in\mathcal{N}_{\mathbb{C}}(0,1))$ are also
uniformly distributed on $[-\pi,\pi),$ we have the following
statement for the functions of the form (\ref{6}).

\begin{Corollary}\sl\label{c1}\sl
There exists an absolute constant $C>0$ such that for the functions
of the form (\ref{6}) we get almost surely
$$
\frac{1}{2\pi}\int_0^{2\pi}\ln|f(re^{i\theta},\omega )|d\theta-\ln|a_k\xi_k(\omega )|
\geq\ln S_f(r,\omega )-(C+1)\ln\ln S_f(r,\omega )
$$
for $r_0(\omega )\leq r<+\infty$ and $k=\min\{n\in\mathbb{Z}_+\colon a_n\neq0\}.$
\end{Corollary}
\begin{proof}[Proof of corollary 4.4.]
It follows from Theorem \ref{tb} that
$
\ln N_g(r,\omega )\leq1+\ln\ln S_g(r)
$
 and by Theorem \ref{ta} we have almost surely
$$
N_g(r,\omega )\geq \ln S_g(r)-C\ln N_g(r,\omega )\geq\ln S_g(r)-(C+1)\ln\ln S_g(r),
$$
for $r_0(\omega )\leq r<+\infty$. Therefore
$$
P\{\omega\colon  (\exists r_0(\omega ))(\forall r>r_0(\omega ))\ [N_g(r,\omega )\geq\ln S_g(r)-(C+1)\ln\ln S_g(r)]\}=1.
$$

Consider a random function
$$
f(z,\omega _1,\omega _2)=\sum_{n=0}^{+\infty}\varepsilon_n(\omega _1)\eta_n(\omega _2)a_nz^n,
$$
where $\varepsilon_n(\omega _1)=e^{i\theta_n(\omega _1)}, \ \theta_n(\omega _1)$
and $\mathop{\rm arg}\eta_n(\omega _2)$ are uniformly distributed on $[-\pi,\pi).$
Also both sequences $\{\varepsilon_n(\omega _1)\},\ \{\eta_n(\omega _2)\}$ are sequences of independent random variables defined on
the Steinhaus probability spaces $(\Omega_1, \mathcal{A}_1, P_1)$ and $(\Omega_2, \mathcal{A}_2, P_2)$,
respectively. Define
\begin{gather*}
A_f=\{(\omega _1,\omega _2)\colon (\exists r_0(\omega _1,\omega _2))(\forall r>r_0(\omega _1,\omega _2))\\
 [N_f(r,\omega _1,\omega _2)\geq\ln S_f(r,\omega _2)-(C+1)\ln\ln S_f(r,\omega _2)]\},
\end{gather*}
where
$$
S^2_f(r,\omega _2)=\sum_{n=0}^{+\infty}|\varepsilon_n(\omega _1)|^2|\eta_n(\omega _2)|^2|a_{n}|^2r^{2n}=
\sum_{n=0}^{+\infty}|\eta_n(\omega _2)|^2|a_{n}|^2r^{2n}.
$$

Consider the events
$$
F=\{\omega_2\colon (\forall n\in\mathbb{N})\ [\eta_n(\omega _2)\neq0]\},\
H=\Bigl\{\omega_2\colon \varlimsup_{n\to+\infty}\sqrt[n]{|a_n||\eta_n(\omega _2)|}=0\Bigl\}.
$$
Then by Lemma \ref{l4} $P_2(H)=1.$ Since $M(\frac{1}{|\xi_n|})<+\infty$, the  probability of the event $F$
$$
1\geq P_2(F)\geq
1-\sum_{n=0}^{+\infty}P_2\{\omega_2\colon  \eta_n(\omega _2)=0\}=1.
$$
Denote $G=F\cap H.$ So, $P_2(G)=1.$
Then for fixed $\omega _2^0\in G$
\begin{gather*}
P_1(A_f(\omega _2^0)){:=}P_1\{\omega _1\colon (\exists r_0(\omega _1,\omega _2^0))(\forall r>r_0(\omega _1,\omega _2^0))\\
  [N_f(r,\omega _1,\omega _2^0)\geq\ln S_f(r,\omega _2^0)-(C+1)\ln\ln S_f(r,\omega _2^0)]\}=1.
\end{gather*}

Let $P$ be a direct product of the probability measures $P_1$ and $P_2$
defined on $(\Omega_1\times\Omega_2, \mathcal{A}_1\times\mathcal{A}_2, P_1\times P_2).$
Here $\mathcal{A}_1\times\mathcal{A}_2$ is the $\sigma$-algebra, which
contains all ${A}_1\times{A}_2$ such that ${A}_1\in\mathcal{A}_1$ and ${A}_2\in\mathcal{A}_2.$
By Fybini's theorem
\begin{gather*}
P(A_f)=\int\limits_{\Omega_2}\Biggl(\int\limits_{A_f(\omega _2)}dP_1(\omega _1)\Biggl)dP_2(\omega _2)\geq
\int\limits_{G}\Biggl(\int\limits_{A_f(\omega _2)}dP_1(\omega _1)\Biggl)dP_2(\omega _2)=\\
=\int\limits_{G}dP_2(\omega _2)=P_2(G)=1.
\end{gather*}

Suppose that there exists a set $B_1\in\Omega_1$ such that $P_1(B_1)>0$
and for all fixed $\omega_1^0\in B_1$
\begin{gather*}
P_2(A_f(\omega _1^0))=P_2\{\omega _2\colon  (\exists r_0(\omega _1^0,\omega _2))(\forall r>r_0(\omega _1^0,\omega _2))\\
[N_f(r,\omega _1^0,\omega _2)\geq\ln S_f(r,\omega _2)-(C+1)\ln\ln S_f(r,\omega _2)]\}<1.
\end{gather*}
Then
\begin{gather*}
P(A_f)=\int\limits_{B_1}\Biggl(\int\limits_{A_f(\omega _1)}dP_2(\omega _2)\Biggl)dP_1(\omega _1)+
\int\limits_{\Omega_2\setminus B_1}\Biggl(\int\limits_{A_f(\omega _1)}dP_2(\omega _2)\Biggl)dP_1(\omega _1)<\\
<\int\limits_{B_1}dP_1(\omega _1)+
\int\limits_{\Omega_2\setminus B_1}dP_1(\omega _1)=1.
\end{gather*}

Contradiction.

Therefore, almost surely by $\omega_1$ we have
$P_2(A_f(\omega _1))=1,$ i.e.
there exists $B_2\in\Omega_1\colon P_1(B_2)=1$
and for all $\omega _1^0\in B_2$ we get
\begin{gather*}
P_2\{\omega _2\colon  (\exists r_0(\omega _1^0,\omega _2))(\forall r>r_0(\omega _1^0,\omega _2))\\ [N_f(r,\omega _1^0,\omega _2)\geq\ln S_f(r,\omega _2)-(C+1)\ln\ln S_f(r,\omega _2)]\}=1.
\end{gather*}

We fix $\omega _1^0\in B_2.$ Let $\varepsilon(\omega _1^0)=\varepsilon^*, \ n\in\mathbb{Z}_+.$
Then for the random entire function
$$
h(z,\omega )=\sum_{n=0}^{+\infty}\varepsilon_n^*\eta_n(\omega )a_nz^n,
$$
we obtain
$$
P\{\omega \colon  (\exists r_0(\omega ))(\forall r>r_0(\omega))\ [N_h(r,\omega )\geq\ln S_h(r,\omega )-(C+1)\ln\ln S_h(r,\omega )]\}=1.
$$

Remark that the sequence $\{\varepsilon_n^*\eta_n(\omega )\}$ is the sequence of independent random variables and sequences
$\{\varepsilon_n^*\eta_n(\omega )\},$ $\{\eta_n(\omega )\}$ are similar.
It remains to denote $\zeta_n(\omega )=\varepsilon_n^*\eta_n(\omega ).$
 \end{proof}

\begin{Theorem}\sl\label{t2}
Let $f$ be a random entire function of the form~(\ref{6})
 such that $a_0\neq0.$
Then there exists $r_0>0$
 such that for all
$r\in(r_0;+\infty)$ we get
$$
p_0(r)\geq s(r)+N(r)\ln N(r)-4N(r).
$$
\end{Theorem}

\begin{proof}[Proof of Theorem \ref{t2}.]
By Jensen's formula we get almost surely
\begin{gather*}
0=\int_0^r\frac{n(t,\omega )}{t}dt=\frac{1}{2\pi}\int_0^{2\pi}
\ln|f(re^{i\theta},\omega )|d\theta-\ln|a_0\xi_0(\omega )|,\\
\ln|a_0\xi_0(\omega )|=\frac{1}{2\pi}\int_0^{2\pi}
\ln|f(re^{i\theta},\omega )|d\theta.
\end{gather*}
Therefore,
$$
P\{\omega\colon  n(r,\omega )=0\}\leq P\Bigl\{\omega\colon  \ln|a_0\xi_0(\omega )|=\frac{1}{2\pi}\int_0^{2\pi}
\ln|f(re^{i\theta},\omega )|d\theta\Bigl\}.
$$
Let
$
G_1=\{\omega\colon \ln|a_0\xi_0(\omega )|\geq\ln\gamma(\omega )\},\
G_2=\Bigl\{\omega\colon \frac{1}{2\pi}\int_0^{2\pi}
\ln|f(re^{i\theta},\omega )|d\theta\leq\ln \gamma(\omega )\Bigl\},
$
where $\gamma(\omega )>1.$ Then
\begin{gather*}
\overline{G_1}\bigcap\overline{G_2}=\Bigl\{\omega\colon \frac{1}{2\pi}\int_0^{2\pi}
\ln|f(re^{i\theta},\omega )|d\theta>\ln \gamma(\omega )>\ln|a_0\xi_0(\omega )|\Bigl\},\\
\overline{G_1}\bigcap\overline{G_2}\subset\Bigl\{\omega\colon \frac{1}{2\pi}\int_0^{2\pi}
\ln|f(re^{i\theta},\omega )|d\theta\neq\ln|a_0\xi_0(\omega )|\Bigl\},\\
G_1\bigcup G_2=\overline{\overline{G_1}\bigcap\overline{G_2}}\supset
\Bigl\{\omega\colon \frac{1}{2\pi}\int_0^{2\pi}
\ln|f(re^{i\theta},\omega )|d\theta=\ln|a_0\xi_0(\omega )|\Bigl\}.
\end{gather*}
So,
\begin{equation}\label{10}
P\{\omega\colon  n(r,\omega )=0\}\leq P(G_1\cup G_2)\leq P(G_1)+P(G_2).
\end{equation}
We define
\begin{gather*}
  A=\Bigl\{\omega\colon  (\exists r_0(\omega ))(\forall r>r_0(\omega ))\ \frac{1}{2\pi}\int_0^{2\pi}\ln|f(re^{i\theta},\omega )|d\theta
\geq\\
\geq\ln S_f(r,\omega )-(C+1)\ln\ln S_f(r,\omega )+\ln|a_0\xi_0(\omega )|\Bigl\}.
\end{gather*}

By Corollary \ref{c1} we obtain that $P(A)=1.$
Put $\gamma(r,\omega )=C_1\cdot|a_0|\cdot|\xi_0(\omega )|,\ C_1>1.$
Then we may calculate the probability of the event $G_1$
$$
P(G_1)=P\{\omega\colon \ln|a_0\xi_0(\omega )|\geq\ln C_1+\ln|a_0\xi_0(\omega )|\}=
P\{\omega\colon \ln C_1\leq0\}=0
$$
and estimate the probability of the event $G_2$ as $r\to+\infty$
\begin{gather}\nonumber
P(G_2)=P(G_2\cap A)+P(G_2\cap \overline{A})\leq
P(G_2\cap A)+P(\overline{A})=P(G_2\cap A)=\\
\nonumber
=P\Bigl\{\omega\colon (\exists r_0(\omega ))(\forall r>r_0(\omega ))\ \Bigl[\ln S_f(r,\omega )-(C+1)\ln\ln S_f(r,\omega )+
\\
\nonumber
+\ln|a_0\xi_0(\omega )|
\leq\frac{1}{2\pi}\int_0^{2\pi}
\ln|f(re^{i\theta},\omega )|d\theta\leq\ln \gamma(r,\omega )\Bigl]\Bigl\}=\\
\nonumber
=P\{\omega\colon (\exists r_0(\omega ))(\forall r>r_0(\omega ))\\
 \nonumber
 \Bigl[\ln S_f(r,\omega )-(C+1)\ln\ln S_f(r,\omega )+\ln|a_0\xi_0(\omega )|\leq\ln C_1
+\ln|a_0\xi_0(\omega )|\Bigl]\}=\\
\nonumber
=P\{\omega\colon (\exists r_0(\omega ))(\forall r>r_0(\omega ))\ \Bigl[\ln S_f(r,\omega )-(C+1)\ln\ln S_f(r,\omega )
\leq\ln C_1\Bigl]\}\leq\\
\nonumber
\leq P\{\omega\colon  (\exists r_0(\omega ))(\forall r>r_0(\omega ))\ \Bigl[\ln S_f(r,\omega )
\leq2\ln C_1\Bigl]\}=\\
\nonumber
=P\{\omega\colon  (\exists r_0(\omega ))(\forall r>r_0(\omega ))\ \Bigl[ S_f(r,\omega)\leq C_1^2\Bigl]\}
= P\{\omega\colon   S_f(r,\omega)\leq C_1^2\}\leq\\
\label{nongaus}
\leq P\Bigl\{\omega\colon  \sum_{n\in\mathcal{N}(r)}|\xi_n(\omega )|^2|a_n|^2r^{2n}\leq C_1^4\Bigl\}.
\end{gather}

The function of distribution of the random variable $|\xi_n(\omega)|$
\begin{gather*}
F_{|\xi_n|}(x)=1-\exp\{-x^2\},\
F_{|\xi_n|^2}(x)=F_{|\xi_n|}(\sqrt{x})=1-\exp\{-x\},\\
F_{|\xi_n|^2|a_n|^2r^{2n}}(x)=F_{|\xi_n|^2}\Bigl(\frac{x}{|a_n|^2r^{2n}}\Bigl)=
1-\exp\Bigl\{-\frac{x}{|a_n|^2r^{2n}}\Bigl\}
\end{gather*}
for $n\not\in\mathcal{N}'$ and $x\in\mathbb{R}_+.$
Then for the random vector $\eta(\omega )=(|\xi_{1}(\omega )|a_{1}r^{j_1},$ $\ldots, \xi_{j_k}(\omega )|a_{j_k}r^{j_k}),\ j_k\in\mathcal{N}(r),$
the density function
$$
p_{\eta}(x)=
\begin{cases}
\prod\limits_{n\in\mathcal{N}(r)}\frac{1}{|a_n|^2r^{2n}}\exp\Bigl\{-\frac{x_n}{|a_n|^2r^{2n}}\Bigl\}, & x\in\mathbb{R}_+^{\mathcal{N}(r)},\\
0 ,& x\not\in\mathbb{R}_+^{\mathcal{N}(r)}.
\end{cases}
$$
So,
\begin{gather}\nonumber
P\Bigl\{\omega\colon \sum_{n\in\mathcal{N}(r)}|\xi_n(\omega )|^2|a_n|^2r^{2n}\leq C_1^4\Bigl\}=
P\{\omega\colon \eta(\omega )\in B(r)\}=
\\
\nonumber
=\prod\limits_{n\in\mathcal{N}(r)}\frac{1}{|a_n|^2r^{2n}}\cdot
\idotsint\limits_{B(r)}\prod\limits_{n\in\mathcal{N}(r)}\exp\Bigl\{-\frac{x_n}{|a_n|^2r^{2n}}\Bigl\}dx_1\ldots dx_k\leq\\
\label{12}
\leq \exp(- s(r))\cdot{\mathop{\rm vol}}_{\mathbb{R}^{{N}(r)}}B(r),
\end{gather}
where
$$
B(r)=\Biggl\{x\in\mathbb{R}_+^{{N}(r)}\colon\sum_{n\in\mathcal{N}(r)}x_n\leq C_1^4\Biggl\}.
$$
For $C>0$ by elementary calculation we get
\begin{gather*}
{\mathop{\rm vol}}_{\mathbb{R}^{n}}\Bigl\{x\in\mathbb{R}^n_+\colon \sum_{i=1}^n x_i\leq C\Bigl\}=\frac{C^n}{n!}.
\end{gather*}
From this equality and Stirling's formula
$$
n!=\sqrt{2\pi n}\Bigl(\frac{n}{e}\Bigl)^n\cdot \exp\Bigl\{-\frac{\theta_n}{12n}\Bigl\}, \
\theta_n\in[0,1],\ n\in\mathbb{N},
$$
it follows that the volume of the set $B(r)$
\begin{gather*}
\ln\Bigl({\mathop{\rm vol}}_{\mathbb{R}^{{N}(r)}}B(r)\Bigl)\leq
-\frac12\ln(2\pi)-\frac12\ln N(r)- N(r)\ln N(r)+\frac1{12 N(r)}+\\
+ N(r)+
4 N(r)\ln C_1\leq-N(r)(\ln N(r)-1-4\ln C_1).
\end{gather*}

Let us choose $C_1=2.$
From (\ref{12}) it follows
$
p_0(r)\geq s(r)+N(r)\ln N(r)-4N(r).
$
\end{proof}
Using Lemma \ref{l3} from Theorems \ref{t1} and \ref{t2} we deduce such a statement.
\begin{Theorem}\sl\label{t3}
Let $\varepsilon>0,$ and $f$ be a random entire function of the form~(\ref{6})
 such that $a_0\neq0.$
Then there exist $r_0>0$ and the set $E\subset(1;+\infty)$ of finite logarithmic measure
 such that for all
$r\in(r_0;+\infty)\setminus E$ we obtain
$$
(1-\varepsilon)N(r)\ln N(r)\leq p_0(r)- s(r)\leq N(r)\exp\{(2+\varepsilon)\sqrt{\ln N(r)}\},
$$
in particular,
\begin{equation}\label{main}
0\leq\varliminf_{\begin{substack} {r\to+\infty \\ r\notin E}\end{substack}}\frac{\ln(p_0(r)- s(r))}{\ln s(r)},\ \varlimsup_{\begin{substack} {r\to+\infty \\ r\notin E}\end{substack}}\frac{\ln(p_0(r)- s(r))}{\ln s(r)}\leq\frac12
\end{equation}
and
$$
\lim_{\begin{substack} {r\to+\infty \\ r\notin E}\end{substack}}\frac{\ln(p_0(r)- s(r))}{\ln N(r)}=1.
$$
\end{Theorem}

\section{Examples on sharpness of inequalities (\ref{main})}
\begin{Theorem}\sl
There exist a random entire function of form (\ref{6}) for which $a_0\neq0$ and
a set $E$ of finite logarithmic measure such that
$$
\lim_{\begin{substack} {r\to+\infty \\ r\notin E}\end{substack}}\frac{\ln(p_0(r)- s(r))}{\ln s(r)}=\frac12.
$$
\end{Theorem}
\begin{proof}[Proof.]
Consider the entire function
$$
f(z)=1+\sum_{n=1}^{+\infty}\frac{z^n}{(\frac{n}{2})^{\frac{n}{2}}}.
$$
For this function and $r\in(r_0(\omega );+\infty)\setminus E$ we have
$$
\frac{\sqrt{ s(r)}}{\ln^3 s(r)}<N(r)<\sqrt{s(r)}\exp\{(1+\varepsilon)\sqrt{\ln s(r)}\},\
\lim_{\begin{substack} {r\to+\infty \\ r\notin E}\end{substack}}\frac{\ln N(r)}{\ln s(r)}=\frac12.
$$
By Theorem \ref{t3} we have for $r\in(r_0;+\infty)\setminus E$
\begin{gather*}
\frac{-\ln2+\ln N(r)+\ln\ln N(r)}{\ln s(r)}\leq\frac{\ln(p_0(r)- s(r))}{\ln s(r)}\leq\frac{\ln N(r)+3\sqrt{\ln N(r)}}{\ln s(r)},\\
\lim_{\begin{substack} {r\to+\infty \\ r\notin E}\end{substack}}\frac{\ln(p_0(r)- s(r))}{\ln s(r)}=
\lim_{\begin{substack} {r\to+\infty \\ r\notin E}\end{substack}}\frac{\ln N(r)}{\ln s(r)}=\frac12.
\end{gather*}
\end{proof}
\begin{Theorem}\sl
There exist a random entire function of form (\ref{6}) for which $a_0\neq0$ and
a set $E$ of finite logarithmic measure such that
$$
\lim_{\begin{substack} {r\to+\infty \\ r\notin E}\end{substack}}\frac{\ln(p_0(r)- s(r))}{\ln s(r)}=0.
$$
\end{Theorem}
\begin{proof}[Proof.]
Consider the entire functions
$$
f(z)=1+\sum_{n=1}^{+\infty}\frac{z^n}{(\frac{n}{2})^{\frac{n}{2}}},\
g(z)=1+\sum_{n\in\mathcal{N}^*}\frac{z^n}{(\frac{n}{2})^{\frac{n}{2}}},
$$
where $\mathcal{N}^*=\{n\colon n=[e^k]+1 \mbox{ for some } k\in\mathbb{Z}_+\}.$
Here $[e^k]$ means the integral part of the real number $e^k.$
We denote
\begin{gather*}
\mathcal{N}_f(r)=\{n\in\mathbb{Z}_+\colon \ln (|a_n|r^n)>0\}\setminus\{0\},\
\mathcal{N}_g(r)=\{n\in\mathcal{N}^*\colon \ln (|a_n|r^n)>0\},
\\
s_f(r)=2\sum_{n\in\mathcal{N}_f(r)}\ln(|a_n|r^n),\
s_g(r)=2\sum_{n\in\mathcal{N}_g(r)}\ln(|a_n|r^n),\ a_n=\Bigl(\frac{n}{2}\Bigl)^{-\frac{n}{2}},\ n\in\mathbb{N}.
\end{gather*}

Remark that the sequence $\{(n/2)^{-n/2}\}$ is log-concave and
$$\mathcal{N}_f(r)=\{1,\ldots,N_f(r)\}.$$
Then by the definition of $N_g(r)$ we get $N_g(r)\leq2\ln N_f(r),\ r\to+\infty.$
For $r\in(r_0;+\infty)\setminus E$
we obtain (see~\cite{nishry 2})
$$
N_g(r)\leq2\ln N_f(r)\leq2\ln(\ln\mu_f(r)\ln^2(\ln\mu_f(r)))<4\ln\ln\mu_f(r).
$$
Remark that
$
\min\{n\in\mathcal{N}'\colon n>\nu_g(r)\}\leq[e\nu_g(r)]+1<(e+1)\ln\nu_g(r).
$
Let us fix $r>0.$
Consider the function $y(t)=\ln(a(t)r^t)=-\frac{t}{2}\ln(\frac{t}{2})+t\ln r,$
for which $a(n)=a_n.$ The graph of the function $y(t)$ passes
through the points $(0;0)$ and $(\nu_g(r),\ln\mu_g(r)).$
It follows from log-concavity of the function $y(t)$ that
the point $(\nu_f(r),\ln\mu_f(r))$ belongs to the
triangle with the vertices $(\nu_g(r),\ln\mu_g(r)),$ $((e+1)\nu_g(r),\ln\mu_g(r))$
and $((e+1)\nu_g(r),(e+1)\ln\mu_g(r)).$ Then
$$
\ln\mu_f(r)\leq(e+1)\ln\mu_g(r),\
s_g(r)\geq2\ln\mu_g(r)\geq\frac{2}{e+1}\ln\mu_f(r).
$$
For the function $g(z)$ and $r\in(r_0;+\infty)\setminus E$
we get
$$
0\leq\lim_{\begin{substack} {r\to+\infty \\ r\notin E}\end{substack}}\frac{\ln(p_0(r)- s_g(r))}{\ln s_g(r)}=
\lim_{\begin{substack} {r\to+\infty \\ r\notin E}\end{substack}}\frac{\ln N_g(r)}{\ln s_g(r)}\leq
\lim_{\begin{substack} {r\to+\infty \\ r\notin E}\end{substack}}\frac{\ln(4\ln\ln \mu_f(r))}{\ln(\frac{2}{e+1}\ln\mu_f(r))}=0.
$$
\end{proof}

\section{Zeros distribution of non-gaussian entire\break functions}
One can ask {\it what happens when random variables $\xi_n(\omega)$ in (\ref{6}) have
not gaussian distribution(\cite{nishry 1})?}
The following result shows that the situation may be very different.

\begin{Theorem}\sl
Let  $f(z,\omega )=\sum_{n=0}^{+\infty}\zeta_n(\omega ) a_nz^n,$
  $a_0\neq0,$ with a sequence of independent identically distributed random variables $(\zeta_n(\omega ))_{n\in\mathbb{Z}_+}$ such that
\begin{itemize}
\item[1)] $(\mathop{\rm arg}\zeta_n(\omega ))_{n\in\mathbb{Z}_+}$ are uniformly distributed on $[-\pi,\pi);$
\item[2)] $M|\zeta_n|<+\infty$ and $ M(\frac1{|\zeta_n|})<+\infty,\ n\in\mathbb{Z}_+;$
\item[3)] there exists $\varepsilon>0$ such that for any $n\in\mathbb{Z}_+$ we have $P\{\omega\colon |\zeta_n(\omega )|<\varepsilon\}=0.$
\end{itemize}
Then there exists $r_0>0$ such that for all $r>r_0$ we get
$$
P\{\omega\colon  n(f,\omega )=0\}=0.
$$
\end{Theorem}
\begin{proof}[Proof]
From inequality (\ref{nongaus}) we get for some constant $C_3>0$
\begin{gather*}
0\leq P\{\omega\colon  n(f,\omega )=0\}\leq P\Bigl\{\omega\colon  \sum_{n\in\mathcal{N}(r)}|\zeta_n(\omega )|^2|a_n|^2r^{2n}\leq C_3\Bigl\}\leq\\
\leq P\Bigl\{\omega\colon  \sum_{n\in\mathcal{N}(r)}\varepsilon^2|a_n|^2r^{2n}\leq C_3\Bigl\}=0\quad (r\to+\infty).
\end{gather*}
\end{proof}

%{\bf Received: Month xx, 200x}

\end{document}